\documentclass[12pt,a4paper]{amsart}
\usepackage{amssymb,latexsym,amsmath,amscd,graphicx,url,color,stmaryrd}
\usepackage[alphabetic]{amsrefs}

\let\oldlabel=\label
\def\prellabel{\marginparsep=1em\marginparwidth=44pt
  \def\label##1{\oldlabel{##1}\ifmmode\else\ifinner\else
    \marginpar{{\footnotesize\ \\ \tt ##1}}\fi\fi}}

\theoremstyle{plain} \newtheorem{thm}{Theorem}[section]
\newtheorem{prop}[thm]{Proposition} 
\newtheorem{Questions}[thm]{Questions} \newtheorem*{Q}{Question}

\theoremstyle{definition} \newtheorem{defn}[thm]{Definition}
\newtheorem{rmk}[thm]{Remark} \newtheorem{ex}[thm]{Example}
\newtheorem*{conditions}{Conditions}

\newcommand{\NN}{{\mathbb N}} \newcommand{\PP}{{\mathbb P}}
\newcommand{\ZZ}{{\mathbb Z}}

 \DeclareMathOperator{\ini}{in}
\DeclareMathOperator{\St}{St} \DeclareMathOperator{\Lex}{Lex}
\DeclareMathOperator{\HS}{HS}  
\title{On Fr\"oberg-Macaulay conjectures for algebras}

\author{M. Boij} \address{KTH - Royal Institute of Technology, SE-100
  44 Stockholm, Sweden } \email{boij@kth.se}

\author{A. Conca} \address{Dipartimento di Matematica, Universit\`a di
  Genova, Via Dodecaneso 35, I-16146 Genova, Italy}
\email{conca@dima.unige.it}

\thanks{The first author was partially supported by VR2013-4545 and
  the second author was partially supported by INdAM-GNSAGA}


\begin{document}
\maketitle

\section*{Introduction}
  Macaulay's theorem and Fr\"oberg's conjecture deal with
  the Hilbert function of homogeneous \emph{ideals} in polynomial rings $S$ over a field $K$.  In this
  short note we present some questions related to variants of
  Macaulay's theorem and Fr\"oberg's conjecture for $K$-\emph{subalgebras}
  of polynomial rings.  In details, given a subspace $V$ of forms of degree $d$ we consider the $K$-subalgebra 
  $K[V]$ of $S$ generated by  $V$.  What can be said about Hilbert function of $K[V]$? The analogy with the ideal case suggests  several  questions. To state them we start by recalling Macaulay's theorem,  Fr\"oberg's conjecture
  and Gotzmann's persistence theorem for ideals. Then we presents the variants for $K$-subalgebras along with some partial results and examples.

\section{Macaulay's theorem and Fr\"oberg's conjecture for ideals} 
Let $K$ be a field and $S=K[x_1,\dots, x_n]$ be the polynomial ring
equipped with its standard grading, i.e., with $\deg x_i=1$ for
$i=1,\dots,n$. Then $S=\bigoplus_{j=0}^\infty S_j$ where $S_j$ denotes
the $K$-vector space of homogeneous polynomials of degree $j$.  Given
positive integers $d,u$ such that $u\leq \dim S_d$ let $G(u,S_d)$ be
the Grassmannian of all $u$-dimensional $K$-subspaces of $S_d$.  For a
given subspace $V\in G(u,S_d)$, the homogeneous components of the
ideal $I=(V)$ of $S$ generated by $V$ are the vector spaces $S_jV$,
i.e., the vector spaces generated by the elements $fg$ with $f\in S_j$
and $g\in V$.
\begin{Q}
  What can be said about the dimension of $S_jV$ in terms of
  $u = \dim V$?
\end{Q}

\subsection{Lower bound:} Macaulay's theorem on Hilbert
functions~\cite{Macaulay} provides a lower bound for $\dim S_j V$
given $\dim V$. It asserts that there exists a subspace
$L\in G(u,S_d)$ such that
\[\dim S_jL \leq \dim S_jV\]
for every $j$ and every $V\in G(u,S_d)$.  Furthermore $\dim S_jL$ can
be expressed combinatorially in terms of $d$ and $u$ by means of the
so-called \emph{Macaulay expansion}, see \cites{BrunsHerzog,Sury} for
details.  The vector space $L$ can be taken to be generated by the
largest $u$ monomials of degree $d$ with respect to the lexicographic
order.  Such an $L$ is called the \emph{lex-segment} (vector space)
associated to the pair $d$ and $u$ and it is denoted by $\Lex(u,S_d)$.

\subsection{Persistence:} A vector space $L\in G(u,S_d)$ is called
\emph{Gotzmann} if it satisfies
\[\dim S_1L =\dim S_1\Lex(u,S_d),\] i.e., if
\[
  \dim S_1L \leq \dim S_1V,
\]
for all $V\in G(u,S_d)$.  Gotzmann's persistence theorem
\cite{Gotzmann} asserts that if $L\in G(u,S_d)$ is Gotzmann then
$S_1L$ is Gotzmann as well.  In particular if $L$ is Gotzmann one has
\[\dim S_jL \leq \dim S_jV,\]
for all $j$ and all $V\in G(u,S_d)$.

\subsection{Upper bound:} Clearly, the upper bound for $\dim S_jV$ is
given by the $\dim S_jW$ for a ``general" $W$ in $G(u,S_d)$. More
precisely, there exists a non-empty Zariski open subset $U$ of
$G(u,S_d)$ such that for every $V\in G(u,S_d)$, for every $j\in \NN$
and every $W\in U$ one has
\[
  \dim S_jV \leq \dim S_jW.
\]

Fr\"oberg's conjecture predicts the values of the upper bound
$\dim S_jW$. For a formal power series
$f(z)=\sum_{i=0}^\infty f_i z^i\in \ZZ\llbracket z\rrbracket$ one
denotes $[f(z)]_+$ the series $\sum_{i=0}^\infty   g_i z^i$, where
$g_i=f_i$ if $f_j\geq 0$ for all $j\leq i$ and $g_i=0$ otherwise.  Given $n,u$ and $d$ one considers the
formal power series:
\[
  \sum c_i z^i=\left [\frac{(1-z^d)^u}{(1-z)^n}\right]_+
\]
and then Fr\"oberg's conjecture asserts that
$\dim S_jW=\dim S_{j+d}-c_{d+j}$ for all $j$.  It is known to be true
in these cases:
\begin{itemize}
\item[(1)] $n\leq 3$ and any $u,d,j$, \cites{Froberg,Anick},
\item[(2)] $u\leq n+1$ and any $d,j$, \cite{Stanley},
\item[(3)] $j=1$ and any $n,u,d$, \cite{HochsterLaksov}
\end{itemize}
and it remains open in general. See \cite{Nenashev} for some recent
contributions.

\section{Macaulay's theorem and Fr\"oberg's conjecture for
   subalgebras}
 For any subspace $V\in G(u,S_d)$ we can consider the $K$-subalgebra
 $K[V]\subseteq S$ generated by $V$. Indeed, $K[V]$ is the coordinate
 ring of the closure of the image of the rational map
 $\PP^{n-1} \dashrightarrow \PP^{u-1}$ associated to $V$.
 
 The homogeneous component of degree $j$ of $K[V]$ is the vector space $V^j$,
 i.e., the $K$-subspace of $S_{jd}$ generated by the elements of the
 form $f_1\cdots f_j$ with $f_1,\dots, f_j\in V$.
 \begin{Q}
   What can be said about the dimension of $V^j$? In other words, what
   can be said about the Hilbert function of the $K$-algebra $K[V]$?
 \end{Q}
 
\begin{defn}
  For positive integers, $n$, $d$, $u$ and $j$, define
  \[
    L(n,d,u,j)=\min\{ \dim V^j : V \in G(u, S_d) \}
  \]
  and
  \[
    M(n,d,u,j)=\max\{ \dim V^j : V \in G(u, S_d) \}.
  \]
\end{defn}
 
\subsection{Lower bound:}  
Recall that a monomial vector space $W$ is said to be \emph{strongly
  stable} if $mx_i/x_j \in W$ for every monomial $m\in W$ and $i<j$
such that $x_j|m$.  Intersections, sums and products of strongly
stable vector spaces are strongly stable. Given monomials
$m_1,\dots,m_c\in S_d$ the smallest strongly stable vector space
containing them is denoted by $\St(m_1,\dots,m_c)$ and it is called
the strongly stable vector space generated by $m_1,\dots,m_c$.
 
\begin{prop} Given $n,d,u$ and $j$ there exists a strongly stable
  vector space $W\in G(u,S_d)$ such that
  \[
    L(n,d,u,j)=\dim W^j
  \] independently of the field $K$.
\end{prop}
 
\begin{proof} Given a term order $<$ on $S$ for every $V\in G(u,S_d)$
  one has $\ini(V)^j\subseteq \ini(V^j)$ for every $j$. Hence one has
  $\dim V_0^j\leq \dim V^j$ where $V_0=\ini(V)$.  Therefore the lower
  bound $L(n,d,u,j)$ is achieved by a monomial vector space.
  Comparing the vector space dimension of monomial algebras is a
  combinatorial problem and hence we may assume the base field has
  characteristic $0$.  Applying a general change of coordinates, we
  may put $V$ in ``generic coordinates" and hence $\ini(V)$ is the
  generic initial vector space of $V$ with respect to some term order.
  Being such it is Borel fixed.  Since the base field has
  characteristic $0$, we have that $\ini(V)$ is strongly
  stable. Therefore the lower bound $L(n,d,u,j)$ is achieved by a
  strongly stable vector space.
\end{proof}
 
\begin{ex} 
  For $n=3,d=4, u=7$ there are $3$ strongly stable vector spaces:
  \[
    \begin{array}{rlcll}
      1)&   \St\{xy^3, x^2z^2\}&=&
      \langle xy^3, x^2z^2, x^2yz, x^2y^2, x^3z, x^3y, x^4\rangle &\mbox{ -- the Lex Segment} \\ 
      2)&   \St\{ xy^2z \}&=&\langle xy^2z, xy^3, x^2yz, x^2y^2, x^3z, x^3y, x^4\rangle \\
      3)&  \St\{y^4, x^2yz \}&=&\langle y^4, xy^3, x^2yz, x^2y^2, x^3z, x^3y, x^4\rangle  &\mbox{  -- the RevLex Segment }
    \end{array}
  \]
  In this case $2)$ and $3)$ turns out to give rational normal scrolls
  of type $(3,2)$ and $(4,1)$ respectively and they give the minimal
  possible Hilbert function in all values.
\end{ex}

\begin{ex}
\label{ex3512}
For $n=3$, $d=5$ and $u=12$, there are five strongly stable
  subspaces of $S_d$:
\[\begin{array}{l}
W_1 = \St\{x^2z^2,xy^3z\}=
    \langle x^5,x^4y,x^4z,x^3y^2,x^3yz,x^3z^2,x^2y^3,x^2y^2z,x^2yz^2,x^2z^3,xy^4,xy^3z\rangle,
   \\
    W_2 =\St\{xy^2z^2\}=
    \langle x^5,x^4y,x^4z,x^3y^2,x^3yz,x^3z^2,x^2y^3,x^2y^2z,x^2yz^2,xy^4,xy^3z,xy^2z^2\rangle,
    \\
     W_3 =\St\{x^2z^3,y^5\}=
    \langle x^5,x^4y,x^4z,x^3y^2,x^3yz,x^3z^2,x^2y^3,x^2y^2z,x^2yz^2,x^2z^3,xy^4,y^5\rangle,
    \\
    W_4 =\St\{x^2z^2y,xy^3z,y^5\}=
    \langle x^5,x^4y,x^4z,x^3y^2,x^3yz,x^3z^2,x^2y^3,x^2y^2z,x^2yz^2,xy^4,xy^3z,y^5\rangle,
    \\
     W_5 =\St\{x^3z^2,y^4z\}= 
    \langle  x^5,x^4y,x^4z,x^3y^2,x^3yz,x^3z^2,x^2y^3,x^2y^2z,xy^4,xy^3z,y^5,y^4z\rangle.
  \end{array}
\]
In this case, neither the Lex segment, $W_1$, nor the RevLex segment,
$W_5$, achieve the minimum Hilbert function. The Hilbert series are
given by
\[
\begin{array}{rl}
\HS_{K[W_1]} (z)  =\HS_{K[W_5]} (z) = & \displaystyle{ \frac{1+9z+3z^2}{(1-z)^3}}, \\
\HS_{K[W_2]} (z)  =\HS_{K[W_4]} (z)  =&  \displaystyle{\frac{1+9z+2z^2}{(1-z)^3}}, \\
\HS_{K[W_3]} (z)  =                                &  \displaystyle{\frac{1+9z+5z^2}{(1-z)^3}}. \\
\end{array}
\]

and the minimum turns out to be $L(3,5,12,j) = \dim W_2^j = \dim W_4^j = 6j^2+5j+1$,
for $j\ge 1$. 
\end{ex}

\begin{Questions}
\label{qstst}
  \begin{itemize}
  \item[(1)] Does there exist a (strongly stable) subspace
    $W\in G(u,S_d)$ such that $L(n,d,u,j)=\dim W^j$ for every $j$?
  \item[(2)] Given $n,d,u,j$ can one characterize combinatorially the
    strongly stable subspace(s) $W$ with the property
    $L(n,d,u,j)=\dim W^j$?
  \item[(3)] Persistence: Assume $W\in G(u,S_d)$ satisfies
    $L(n,d,u,2)=\dim W^2$. Does it satisfies also
    $L(n,d,u,j)=\dim W^{j}$ for all $j$?
  \end{itemize}
\end{Questions}

\begin{rmk} 
For $n=2$  there exists only one strongly stable vector space in $G(u,S_d)$, i.e. 
$\langle x^{d}, x^{d-1}y, \dots, x^{d-u+1}y^{u-1} \rangle$ (which is both the Lex and RevLev segment) and the questions in \ref{qstst} have all straightforward answers. 
\end{rmk}

\begin{rmk} 
It is proved in  \cite{DN}  that Lex-segments, RevLex-segments and principal strongly stable vector spaces define  
normal  and Koszul toric rings (in particular Cohen-Macaulay). Furthermore in \cite{DFMSS} it is proved that a 
strongly stable vector spaces with two strongly stable generators define a Koszul toric ring. 
On the other hand, there are examples of strongly stable vector spaces  with a non-Cohen-Macaulay and non-quadratic  toric ring, see \cite[Example 1.3]{BC}. 
\end{rmk}

\subsection{Upper bound:}    
As in the ideal case, the upper bound is achieved by a general
subspace $W$, i.e., for $W$ in a non-empty Zariski open subset of
$G(u,S_d)$.
\begin{Q}
  What can be said about the value $M(n,d,u,j)$?
\end{Q}
Obviously,


 \begin{equation}
 \label{eq1} 
  M(n,d,u,j)\leq \min\left\{ \dim S_{jd}, \binom {u-1+j}{u-1} \right\}
\end{equation}

and the naive expectation is that equality holds in  (\ref{eq1}), i.e., if
$f_1, \dots, f_u$ are general forms of degree $d$, then the monomials
of degree $j$ in the $f_i$'s are either linearly independent or they
span $S_{jd}$. It turns out that in nature things are more complex
than expected at first. First of all, if $u>n$ then equality in (\ref{eq1}) would imply that 
for a generic $W$ one would have $W^j=S_{jd}$ for large $j$. This  fact can be stated  
in terms of projections of the $d$-th Veronese variety: the projection associated to $W$ is an isomorphism.  
Recall that a generic linear projection of a smooth projective variety of dimension $m$  from some  projective space where its embedded, into a projective space of dimension $c$ is an isomorphism if  $c\geq 2m+1$. 
Hence we have that if $u\geq 2n$ then equality in  (\ref{eq1}) holds at least for large $j$. On the other hand, for $n+1\leq u<2n$ equality in  (\ref{eq1}) should not be expected unless one knows  that the corresponding projection of the Veronese variety behaves in an unexpected way.  

Summing up, the most natural question turns out to be: 

\begin{Q} 
Assume that $u\geq 2n$.  Is it true that 
\[ 
  M(n,d,u,j)=\min\left\{ \dim S_{jd}, \binom {u-1+j}{u-1} \right\}
\]
holds for all $j$? 
\end{Q} 

The answer turns out to be negative as the following example shows:

\begin{prop}
\label{8quadrics} 
 For any space $W$
  generated by eight quadrics in four variables the dimension of $W^2$
  is at most $34$ independently of the base field $K$. That is:
  $$M(4,2,8,2)\leq 34<\min\left\{\dim S_{4},\binom{7+2}{7}\right\}=35.$$
\end{prop}

\begin{rmk}
  This assertion was proven in \cite{COR}*{Theorem.~2.4} using a
  computer algebra calculation. Here we present a more conceptual
  argument.
\end{rmk}

\begin{proof} Firstly we may assume   that $K$ has characteristic $0$ and is algebraically
  closed.  Secondly we may assume that  $W$ is generic. The $8$-dimension space of quadrics $W$
  is apolar to a $2$-dimension space of quadrics, call it $V$. A pair
  of generic quadrics can be put simultaneously in diagonal form,
  i.e., that $V$ is generated by $x_1^2+x_2^2+x_3^2+x_4^2$ and
  $a_1x_1^2+a_2x_2^2+a_3x_3^2+a_4x_4^2$. See for example
  \cite{Wonenburger}. Indeed, it is sufficient that $V$ contains a
  quadric of rank $4$ since that can be put into the form
  $x_1^2+x_2^2+x_3^2+x_4^2$ and the other form can then be
  diagonalized preserving the first.
As a consequence, after a change of coordinates $W$ contains $x_ix_j$ with $1\leq i<j\leq
  4$. Since $(x_1x_4)(x_2x_3)=(x_1x_2)(x_3x_4)$ and
  $(x_1x_3)(x_2x_4)=(x_1x_2)(x_3x_4)$ we have at least two independent
  relations among the $36$ generators of $W^2$. Therefore
  $\dim W^2\leq 34$.
\end{proof}
 
More precisely one has:
\begin{prop}
 One has  $M(4,2,8,2)=34$ independently of the base field $K$.
\end{prop}
\begin{proof} We have already argued that $M(4,2,8,2)\leq
  34$. Therefore it is enough to describe an $8$-dimension space of
  quadrics $W$ in $4$ variables such that $\dim W^2=34$.  We set
  \[
    W_0=\langle x_ix_j : 1\leq i<j\leq 4\rangle
  \]
  and
  \[
    F=a_1x_1^2+a_2x_2^2+a_3x_3^2+a_4x_4^2 \mbox{ and }
    G=b_1x_1^2+b_2x_2^2+b_3x_3^2+b_4x_4^2
  \]
  Then we set $W_1=\langle F,G\rangle$ and then
  \[W=W_0+W_1.\] We consider two conditions on the coefficients $a_1,a_2,\dots, b_4$:
 
\begin{conditions}
  (1) All the $2$-minors of
  \[
    \begin{pmatrix}
      a_1& a_2& a_3& a_4\\ b_1& b_2& b_3& b_4
    \end{pmatrix}
  \]
  are non-zero.

  (2) The matrix
  \[
    \begin{pmatrix}
      a_1^2&a_2^2&a_3^2&a_4^2\\
      b_1^2&b_2^2&b_3^2&b_4 ^2\\
      a_1b_1&a_2b_2&a_3b_3&a_4b_4\\
    \end{pmatrix}
  \]
  has rank $3$.
\end{conditions}
We observe that $W_0^2$ is generated by the $19$ monomials of degree
$4$ and largest exponent $\leq 2$. Then we note that if $W_1$ contains
a quadric $q$ supported on $x_i^2,x_j^2$ and $x_h^2$ with $i,j,h$
distinct and $k\not\in\{ i,j,h\}$ then $x_kx_iq=x_kx_i^3 \mod(W_0^2)$
and similarly for $j$ and $h$.  This implies that if Condition (1)
holds then $W_0^2+W_0W_1$ is generated by the $31$ monomials different
from $x_1^4,\dots, x_4^4$.  Assuming that Condition (1) holds, we have that
the matrix representing $F^2, G^2, FG$ in $S_4/ W_0^2+W_0W_1$ is
exactly the one appearing in Condition (2). Then are $F^2, G^2, FG$
are linearly independent mod $W_0^2+W_0W_1$ if  and only if
Condition (2) holds.  Summing up, if Conditions (1) and (2) hold then
$\dim W^2=34$.   Finally we observe that  for $F=x_1^2+x_3^2+x_4^2$ and
$G=x_2^2+\alpha x_3^2+x_4^2$ the conditions (1) and (2) are satisfied provided  $\alpha\neq 0$ 
and $\alpha\neq 1$. Hence this (conceptual) argument works for any field but $\ZZ/2\ZZ$. Over $\ZZ/2\ZZ$ one can consider the space $W$ generated by 
$x_{1}^2, x_{2}^2, x_{3}^2, x_{1}x_{3}, x_{2}x_{4}, x_{3}x_{4}, x_{2}x_{3} + x_{1}x_{4}, x_{1}x_{2} + x_{4}^2$ and check with the help of a computer algebra system that $\dim W^2=34$. 
 \end{proof} 
 
 As far as we know the case discussed  in  Proposition  \ref{8quadrics} is the only known case where $u\geq 2n$ and the actual value of  $M(n,d,u,j)$ is smaller than $\min\left\{ \dim S_{jd}, \binom {u-1+j}{u-1} \right\}$.

 
{\bf Acknowledgements:} We thank Winfred Bruns and David Eisenbud for helpful discussions on the material of this paper.

\end{document}